\pdfoutput=1
\documentclass[acmsmall,natbib=false,nonacm,screen]{acmart}
\setcopyright{none}
\AtBeginDocument{%
  }

\RequirePackage[
  datamodel=acmdatamodel,
  style=acmauthoryear,
  date=year,
]{biblatex}

\usepackage[appendix=inline,bibengine=biblatex,bibliography=common]{apxproof}

\newtheoremrep{theorem}{Theorem}

\usepackage{subcaption}
\usepackage{enumitem}
\usepackage{mathtools}

\usepackage{xcolor-solarized}
\colorlet{yellow}{solarized-yellow}
\colorlet{orange}{solarized-orange}
\colorlet{red}{solarized-red}
\colorlet{magenta}{solarized-magenta}
\colorlet{violet}{solarized-violet}
\colorlet{blue}{solarized-blue}
\colorlet{cyan}{solarized-cyan}
\colorlet{green}{solarized-green}

\colorlet{nilbg}{solarized-cyan!25}
\colorlet{consbg}{solarized-violet!25}

\colorlet{Appbg}{solarized-blue!25}
\colorlet{itbg}{solarized-magenta!25}

\colorlet{nbg}{solarized-green!25}
\colorlet{cbg}{solarized-red!25}

\colorlet{fbg}{solarized-magenta!25}
\colorlet{fsbg}{solarized-cyan!25}
\colorlet{gbg}{solarized-magenta!25}
\colorlet{hbg}{solarized-magenta!25}

\colorlet{etabg}{solarized-violet!25}

\colorlet{ebg}{solarized-yellow!25}
\colorlet{mbg}{solarized-green!25}

\usepackage{rv-macros}

\usepackage{quiver}
\newcommand{\fsLemma}[1]{\citeauthor{fioreListObjectsAlgebraic2017}~\citeyear{fioreListObjectsAlgebraic2017}, Lemma~#1}

\begin{document}

\title{Symmetric List Objects}

\author{Fiona Blackett}
\orcid{0009-0001-2966-819X}
\email{fiona.blackett@strath.ac.uk}
\affiliation{%
  \institution{University of Strathclyde}
  \country{UK}}

\author{Vikraman Choudhury}
\orcid{0000-0003-2030-8056}
\email{vikraman.choudhury@strath.ac.uk}
\affiliation{%
  \institution{University of Strathclyde}
  \country{UK}}

\author{Rin Liu}
\orcid{0009-0005-8791-2495}
\email{rin.liu@strath.ac.uk}
\affiliation{%
  \institution{University of Strathclyde}
  \country{UK}}

\begin{abstract}
  Cons-lists are fundamental to functional programming, and are well-understood as initial algebras, and as a concrete construction of free monoids. In category theory, such list objects can be axiomatised internal to any monoidal category, with nil and cons maps, and a recursion mechanism. When strengthened to support parametrised recursion, they can be shown to give the free monoid construction. We study a small variant on parametrised list objects called parametrised symmetric list objects, in a symmetric monoidal category, where the two frontmost elements are allowed to be permuted. We show that this gives a construction of free commutative monoids, and develop some related theory.
\end{abstract}

\maketitle

\section{Introduction}
\label{sec:introduction}

Lists are a fundamental data structure, important in computer science and mathematics. They are often defined in functional programming languages as cons-lists, generated inductively from a pair of constructors $\nil:L(A)$ and $\cons:A\to L(A)\to L(A)$, and whose eliminators allow for definitions by recursion and proofs by induction.

In category theory, these lists are understood as initial algebras $\mu X.1+A\times X$~\cite{DBLP:phd/ethos/Hagino87,adamekInitialAlgebrasTerminal2025}. This list structure can be axiomatised more generally in the setting of a monoidal category, replacing the cartesian product with a tensor, and removing the need for a coproduct by axiomatising the $\nil$ and $\cons$ maps directly with an iteration mechanism. These list objects generalise the notion of natural numbers objects~\cite{lawvereFUNCTORIALSEMANTICSALGEBRAIC1963} and were introduced in Joyal's arithmetic universes~\cite[see][]{maiettiJoyalsArithmeticUniverses2003}. An internal language for them was studied by \textcite{maiettiJoyalsArithmeticUniverse2010}, and other variants were explored by \textcite{cockettListarithmeticDistributiveCategories1990} and \textcite{jayFiniteObjectsLocos1997}.

In all of these settings, additional axioms on the underlying category are required to reconstruct traditional list operations, such as list append ($\App$), from these list objects. However, in the minimal setting of a monoidal category, this notion of a list object can be strengthened to that of a \textit{parametrised} list object which supports parametrised recursion. It is shown by \textcite{fioreListObjectsAlgebraic2017} that this yields a construction of free monoids~\cite[see also][]{kellyUnifiedTreatmentTransfinite1980} in monoidal categories. We give a string diagrammatic proof of this result. As a consequence~\cite{fioreListObjectsAlgebraic2017}, if a monoidal category $\cat{C}$ has all parametrised list objects, then the forgetful functor $\mathsf{Mon}(\cat{C})\to\cat{C}$ is monadic.

Finite multisets are another important data structure that are the commutative counterpart to lists, and in fact can be axiomatised with only a simple modification to these list objects. In this paper, we study the notion of a \textit{symmetric} list object in a symmetric monoidal category. We identify lists that differ by a transposition of the two frontmost elements using the symmetry of the underlying category. Using this, we show the analogous commutative results that symmetric list objects are free commutative monoids, and that having all symmetric list objects yields a similar monadic adjunction between $\mathsf{CMon}(\cat{C})$ and $\cat{C}$.

In type theory, finite multisets have been considered as setoids over lists, where the equivalence relation is bag equivalence~\cite{danielssonBagEquivalenceProofRelevant2012,altenkirchDefinableQuotientsType}, or as higher inductive types of lists with an additional swap constructor to permute the two head elements~\cite{choudhuryFreeCommutativeMonoids2023}. These constructions are carried out in type theories with function types, whereas our setting only requires a symmetric monoidal structure without closure.

\section{List objects}
\label{sec:list-objects}

Let $(\cat{C},\otimes,1)$ be a strict monoidal category. By coherence, strictness is no loss of generality, and we freely treat cartesian monoidal categories as strict.

\begin{definition}\label{def:list-object}
  A (\textit{parametrised}) \textit{list object} (hereafter, just \textit{list object}) $L(A)$ on an object $A$ consists of a pair of maps $\nil:1\to L(A)$ and $\cons:A\otimes L(A)\to L(A)$
  such that for all objects $P$ and $L$ equipped with maps $n:P\to L$ and $c:A\otimes L\to L$,
  there exists a unique \textit{mediating map} or \textit{iterator} $\iter(n,c):L(A)\otimes P\to L$
  satisfying computation rules (a) and (b):

  \begin{equation*}

  \end{equation*}
\end{definition}

\begin{toappendix}

  When rewriting a diagram, the affected portions will be outlined, and the results on the following diagram will be outlined with a dashed border. String labels will also be omitted from diagrams in proofs to reduce clutter, with colours indicating the types.

  We define the append operation $\App:L(A)\otimes L(A)\to L(A)$ as the iterator $\iter(\id,\cons)$, and the unit map $\eta:A\to L(A)$ as:
  \begin{equation*}
    %
  \end{equation*}

  We note that $\App$ computes on $\eta$ in the left argument as shown in the following lemma.
      \qedhere
    \end{equation*}
  \end{proof}
\end{toappendix}

\begin{theoremrep}\label{thm:free-monoid}
  The list object $L(A)$ is the free monoid on $A$.
\end{theoremrep}
\begin{proof}
  We claim that $(L(A),\nil,\App)$ is a free monoid structure on $A$.

  The left unit law directly holds by $\App$-$\nil$-$\beta$:
  \begin{equation*}
    %
%
  \end{equation*}

  For the right unit law, consider the following maps $f,g:L(A)\to L(A)$ below:
  \begin{equation*}
    %
  \end{equation*}

  We claim that these maps are both the iterator $\iter(\nil,\cons)$. By the uniqueness of the mediating map, it suffices to show that the two maps satisfy the required computation rules.

  The identity map trivially satisfies both rules, so we have $g=\iter(\nil,\cons)$. We have $f$-$\nil$-$\beta$ by applying $\App$-$\nil$-$\beta$:
  \begin{equation*}
    %
  \end{equation*}

  and hence $f=\iter(\nil,\cons)=g$.

  For associativity, let $f,g:L(A)\otimes L(A)\otimes L(A)\to L(A)$ be the two maps below:
  \begin{equation*}
    %
  \end{equation*}

  We claim that both maps are equal to $\iter(\App,\cons)$. We have $\nil$-$\beta$ for both since $\App$ computes on $\nil$ in the left argument. In detail, for $f$:
  \begin{equation*}
    %
  \end{equation*}

  Similarly, we have $\cons$-$\beta$ since $\App$ also computes on $\cons$. For $f$:
  \begin{equation*}
    %
  \end{equation*}

  and for $g$, we push the $\cons$ through $\App$ with two computations:
  \begin{equation*}
    %
  \end{equation*}

  Hence $f=\iter(\App,\cons)=g$.

  It remains to show the unique extension property.
  Let $f:A\to B$ be a morphism, with $(B,e,m)$ a monoid object.
  We define the extension $f^\sharp:L(A)\to B$ as the mediating map $\iter(n,c)$,
  where $n:1\rightarrow B$ and $c:A\otimes B\rightarrow B$ are:
  \begin{equation*}
    %
  \end{equation*}

  We verify that $f^\sharp\circ\eta=f$:
  \begin{equation*}
    %
  \end{equation*}

  Next, we check that $f^\sharp$ is a monoid homomorphism.

  Since $f^\sharp=\iter(e,\ldots)$, $f^\sharp$ computes on $\nil$ to $e$ and hence preserves the monoid identity. For the homomorphism property, let $g,h:L(A)\otimes L(A)\to B$ be the following maps:
  \begin{equation*}
    %
  \end{equation*}

  We show that these maps are both the iterator $\iter(f^\sharp,c)$, where $c$ is the same list algebra map as in the definition of $f^\sharp$ above. We have $g$-$\nil$-$\beta$ by $\App$-$\nil$-$\beta$, and $h$-$\nil$-$\beta$ by $f^\sharp$-$\nil$-$\beta$ and left unit for $B$:
  \begin{equation*}
    %
  \end{equation*}

  For $\cons$-$\beta$ we have the computations:
  \begin{equation*}
    %
  \end{equation*}

  Finally, we show that the extension $f^\sharp$ is unique. Suppose there is another monoid homomorphism $g:L(A)\to B$ satisfying $g\circ\eta=f$. Then, $g$ satisfies the same computation rules as $f^\sharp$. Indeed, $\nil$-$\beta$ holds since $g$ preserves identities, and $\cons$-$\beta$ holds via the following computation:
  \begin{equation*}
    %
    \qedhere
  \end{equation*}
\end{proof}

\begin{theorem}[\fsLemma{3.4.2}]
  If $\cat{C}$ admits parametrised list objects for every $A$,
  then $\mathsf{Mon}(\cat{C})$ is monadic over $\cat{C}$ via the evident forgetful functor.
\end{theorem}
\begin{proof}[Proof sketch]
  By Theorem~\ref{thm:free-monoid}, $A\mapsto L(A)$ is left adjoint to the forgetful functor, giving a monad $L$ with unit $\eta$, action on maps $L(f)=(\eta\circ f)^\sharp$, and multiplication $\mu_A=\id_{L(A)}^\sharp:LL(A)\to L(A)$. We show that the comparison functor $\mathsf{Mon}(\cat{C})\to L$-$\mathsf{Alg}$, sending $(B,e,m)$ to $\id_B^\sharp:L(B)\to B$, is an equivalence. The unit law holds by construction of the extension, as $\id_B^\sharp\circ\eta=\id_B$. Both $\id_B^\sharp\circ\mu_B$ and $\id_B^\sharp\circ L(\id_B^\sharp)$ are monoid homomorphisms $LL(B)\to B$ extending $\id_B^\sharp$, so the multiplication law holds by uniqueness of extensions.

  Conversely, given an algebra $\alpha:L(A)\to A$, let $e=\alpha\circ\nil$ and $m=\alpha\circ\cons\circ(A\otimes\eta)$, which by Lemma~\ref{lem:eta-app} is $\alpha\circ\App\circ(\eta\otimes\eta)$. Since $\mu_A$ is a monoid homomorphism with $\mu_A\circ\eta_{L(A)}=\id$, we have $\App_A=\mu_A\circ\App_{L(A)}\circ(\eta_{L(A)}\otimes\eta_{L(A)})$, so the algebra law $\alpha\circ\mu_A=\alpha\circ L(\alpha)$, the fact that $L(\alpha)$ is a monoid homomorphism, and naturality of $\eta$ give $\alpha\circ\App_A=m\circ(\alpha\otimes\alpha)$, so $\alpha$ is a split epimorphism (split by $\eta$) that preserves the unit and multiplication, and the monoid laws for $(A,e,m)$ follow from those of $(L(A),\nil,\App)$. In particular, $\alpha\circ\cons=m\circ(A\otimes\alpha)$, so $\alpha$ satisfies the computation rules of $\iter(e,m)=\id_A^\sharp$, and $\alpha=\id_A^\sharp$ by uniqueness of the mediating map. Starting from a monoid, $\id_B^\sharp\circ\nil=e$ and $\id_B^\sharp\circ\cons\circ(B\otimes\eta)=m\circ(B\otimes\id_B^\sharp\circ\eta)=m$, so the two constructions are mutually inverse. Finally, a map is a monoid homomorphism if and only if it is an algebra homomorphism, again by uniqueness of extensions.
\end{proof}

\begin{theoremrep}\label{thm:strong-up}
  If $(\cat{C},\times,1)$ is cartesian monoidal and $L(A)$ is a list object, then $L(A)$ satisfies the following stronger universal property: for all objects $P$ and $L$ equipped with maps $n:P\to L$ and $c:A\times L\times P\rightarrow L$, there exists a unique \textit{doubly parametrised iterator} $\checkiter(n,c):L(A)\times P\rightarrow L$, such that the following two computation rules hold:
  \begin{equation*}
    \begin{tikzpicture}[baseline=(it)]
      \begin{scope}
        \node[operator,fill=itbg] (it) at (0,0) {$\checkiter(n,c)$};
        \node[anchor=east,operator,fill=nilbg] (in1) at ($(it.140)+(-1,0.4)$) {$\nil$};
        \node (in2) at ($(it.220)+(-1,-0.4)$) {};
        \node (out) at ($(it.east)+(1,0)$) {};
        \draw[blue] (in1) to[out=0,in=140] node[inputwire,pos=0.4] {$L(A)$} (it.140);
        \draw[green] ($(in2.west) - (0.25,0)$) to[out=0,in=220] node[inputwire,below] {$P$} (it.220);
        \draw[cyan] (it) -- (out.west) node[outputwire] {$L$};
      \end{scope}

      \node at (2,0) {$=$};
      \node at (2,-2) {\textup{(a) }$\nil$-$\beta$};

      \begin{scope}[xshift=3.65cm]
        \node[operator,fill=nbg] (n) at (0,0) {$n$};
        \node (in) at ($(n.west)+(-1,0)$) {};
        \node (out) at ($(n.east)+(1,0)$) {};
        \draw[green] (in) -- node[inputwire] {$P$} (n);
        \draw[cyan] (n) -- node[outputwire] {$L$} (out);
      \end{scope}
    \end{tikzpicture}%
    \quad
    \begin{tikzpicture}[baseline=(base)]
      \begin{scope}
        \node (base) at (0,0) {};
        \node[operator,fill=itbg] (it) at (0,0) {$\checkiter(n,c)$};
        \node[anchor=east,operator,fill=consbg] (cons) at ($(it.140)+(-1,0.4)$) {$\cons$};

        \coordinate (in1) at ($(cons.140)+(-1,0.2)$);
        \coordinate (in2) at ($(cons.220)+(-1,-0.2)$);
        \coordinate (in3) at ($(in1 |- it.220)+(0,-0.6)$);
        \coordinate (out) at ($(it.east)+(1,0)$);

        \draw[blue] (cons) to[out=0,in=140] node[inputwire,pos=0.4] {$L(A)$} (it.140);
        \draw[green] (in3) to[out=0,in=220] node[inputwire,pos=0.1,below] {$P$} (it.220);

        \draw[cyan] (it) -- (out.west) node[outputwire] {$L$};
        \draw[magenta] (in1) to[out=0,in=140] node[inputwire] {$A$} (cons.140);
        \draw[blue] (in2) to[out=0,in=220] node[inputwire,below] {$L(A)$} (cons.220);
      \end{scope}

      \node at (2,0) {$=$};
      \node at (2,-2) {\textup{(b) }$\cons$-$\beta$};

      \begin{scope}[xshift=6.2cm]
        \node[operator,fill=cbg] (c) at (0,0) {$c$};
        \node[anchor=east,operator,fill=itbg] (it) at ($(c.220)+(-1,-0.4)$) {$\checkiter(n,c)$};

        \coordinate (in2) at ($(it.140)+(-1.5,0.2)$);
        \coordinate (in1) at ($(in2 |- c.140)+(0,0.8)$);
        \coordinate (in3) at ($(in1 |- it.220)+(0,-0.2)$);
        \node[circle,inner sep=1pt,fill=black] (anchor) at ($(in3)+(1,0)$) {};
        \coordinate (out) at ($(c.east)+(1,0)$);

        \draw[magenta] (in1) to[out=0,in=140] node[inputwire,pos=0.1] {$A$} (c.140);
        \draw[blue] (in2) to[out=0,in=140,in looseness=0.4] node[inputwire] {$L(A)$} (it.140);
        \draw[green] (in3) to[out=0,in=180] node[inputwire,below] {$P$} (anchor);
        \draw[green] (anchor) to[out=340,in=220,out looseness=1.5] (c.220);
        \draw[green] (anchor) to[out=40,in=220] (it.220);
        \draw[cyan] (it) to[out=0,in=180,in looseness=1.2] node[inputwire,pos=0.4,xshift=-2pt] {$L$} (c.180);

        \draw[cyan] (c) -- (out) node[outputwire] {$L$};
      \end{scope}
    \end{tikzpicture}%
  \end{equation*}
\end{theoremrep}

\begin{proof}[Proof sketch]
  Consider the list algebra structure on $K=L\times P$ given by $\check{n}:P\xrightarrow{\Delta_P}P\times P\xrightarrow{n\times P}L\times P$ and $\check{c}:A\times L\times P\xrightarrow{A\times L\times\Delta_P}A\times L\times P\times P\xrightarrow{c\times P}L\times P$. By the ordinary universal property, we have an iterator $\iter(\check{n},\check{c}):L(A)\times P\to L\times P$. Then, $\checkiter(n,c)$ is given by $\pi_1\circ\iter(\check{n},\check{c})$. Its computation rules follow from those of $\iter(\check{n},\check{c})$, using that $\pi_2\circ\iter(\check{n},\check{c})=\pi_2$, which holds since both sides are the iterator $\iter(\id_P,\pi_2)$. For uniqueness, if $h$ satisfies the computation rules, then $\langle h,\pi_2\rangle$ satisfies those of $\iter(\check{n},\check{c})$, so $h=\pi_1\circ\iter(\check{n},\check{c})$.
\end{proof}

\begin{theoremrep}[\fsLemma{3.4.3}]
  If $(\cat{C},\times,1)$ is cartesian monoidal and admits list objects for every $A$, then the induced monad $L:\cat{C}\to\cat{C}$ is strong.
\end{theoremrep}

\begin{proof}[Proof sketch]
  Since the category is symmetric monoidal, we need only define one of the strengths. A construction of the right strength $\tau_{A,B}:L(A)\times B\to L(A\times B)$ is as follows; the left strength $\sigma_{A,B}:A\times L(B)\to L(A\times B)$ is then obtained by conjugating $\tau_{B,A}$ with the symmetry.

  We can equip $L(A\times B)$ with a doubly parametrised list algebra structure with the maps
  \begin{align*}
    n & \coloneqq B\xrightarrow{!}1\xrightarrow{\nil}L(A\times B)                                                              \\
    c & \coloneqq A\times L(A\times B)\times B\xrightarrow{\cong} A\times B\times L(A\times B)\xrightarrow{\cons} L(A\times B)
  \end{align*}
  and from the strong universal property (Theorem~\ref{thm:strong-up}), we obtain a map $\checkiter(n,c):L(A)\times B\rightarrow L(A\times B)$. This iterator is the right strength map.

  The required strength coherences then follow from the uniqueness part of the strong universal property, by checking that both sides of each coherence satisfy the same computation rules.
\end{proof}

\section{Symmetric List Objects}

Hereafter, we assume $\cat{C}$ is symmetric strict monoidal, i.e. strict monoidal with a symmetric braiding.

\begin{definition}
  A (\textit{parametrised}) \textit{symmetric list object} or \textit{ilst} $L(A)$ on an object $A$ consists of a pair of maps $\nil:1\to L(A)$ and $\cons:A\otimes L(A)\to L(A)$ satisfying the following \textit{ilst axiom}:
  \begin{equation*}
    \begin{tikzpicture}
      \begin{scope}
        \node[operator,fill=consbg] (cons2) at (0,0) {$\cons$};
        \node[operator,fill=consbg] (cons1) at ($(cons2.220)+(-1.3,-0.4)$) {$\cons$};

        \coordinate (in2) at ($(cons1.140)+(-1,0.2)$);
        \coordinate (in1) at ($(in2 |- cons2.140)+(0,0.6)$);
        \coordinate (in3) at ($(in1 |- cons1.220)+(0,-0.2)$);
        \coordinate (out) at ($(cons2.east)+(1,0)$);

        \draw[magenta] (in1) to[out=0,in=140] node[inputwire,pos=0.05] {$A$} (cons2.140);
        \draw[magenta] (in2) to[out=0,in=140] node[inputwire,pos=0.15] {$A$} (cons1.140);
        \draw[blue] (in3) to[out=0,in=220] node[inputwire,below] {$L(A)$} (cons1.220);
        \draw[blue] (cons1) to[out=0,in=220] node[below,pos=0.4] {$L(A)$} (cons2.220);
        \draw[blue] (cons2) -- node[outputwire] {$L(A)$} (out);
      \end{scope}

      \node at (2,0) {$=$};
      \node at (2,-2) {ilst axiom};

      \begin{scope}[xshift=5.5cm]
        \node[operator,fill=consbg] (cons2) at (0,0) {$\cons$};
        \node[operator,fill=consbg] (cons1) at ($(cons2.220)+(-1.3,-0.4)$) {$\cons$};

        \coordinate (in2) at ($(cons1.140)+(-1,0.2)$);
        \coordinate (in1) at ($(in2 |- cons2.140)+(0,0.6)$);
        \coordinate (in3) at ($(in1 |- cons1.220)+(0,-0.2)$);
        \coordinate (out) at ($(cons2.east)+(1,0)$);

        \draw[magenta] (in1) to[out=0,in=140] node[inputwire,pos=0.1] {$A$} (cons1.140);
        \draw[magenta] (in2) to[out=0,in=140] node[inputwire,pos=0.05] {$A$} (cons2.140);
        \draw[blue] (in3) to[out=0,in=220] node[inputwire,below] {$L(A)$} (cons1.220);
        \draw[blue] (cons1) to[out=0,in=220] node[below,pos=0.4] {$L(A)$} (cons2.220);
        \draw[blue] (cons2) -- node[outputwire] {$L(A)$} (out);
      \end{scope}
    \end{tikzpicture}%
  \end{equation*}
  such that for all objects $P$ and $L$ equipped with maps $n:P\to L$ and $c:A\otimes L\to L$, where $c$ also satisfies the ilst axiom (with $L$ in place of $L(A)$), there exists a unique iterator $\iter(n,c):L(A)\otimes P\to L$ satisfying the computation rules (a) and (b) of Definition~\ref{def:list-object}.
\end{definition}

Since the universal property only quantifies over list algebras satisfying the ilst axiom, a symmetric list object need not be a list object: for instance, in $\mathbf{Set}$, finite multisets on a set $A$ form a symmetric list object, which is not a list object when $A$ has at least two elements. Conversely, every list object satisfying the ilst axiom is a symmetric list object. For example, the list object $L(1)$, a natural numbers object, is a symmetric list object on $1$, since the symmetry on $1\otimes1$ is the identity. As $\cons$ satisfies the ilst axiom, $\App=\iter(\id,\cons)$ and $\eta$ are also defined for symmetric list objects, and Lemma~\ref{lem:eta-app} holds with the same proof.

\begin{toappendix}

  \begin{lemma}\label{lem:eta-app-braided}
    For any symmetric list object $L(A)$:
    \begin{equation*}

    \end{equation*}
  \end{lemma}
  \begin{proof}
    Let the two maps be $f$ and $g$, respectively. We claim that both maps are the iterator $\iter(\eta,\cons)$.

    $\nil$-$\beta$:
    \begin{equation*}
      %
      \qedhere
    \end{equation*}
  \end{proof}

\end{toappendix}

\begin{theoremrep}
  Let $(L(A),\nil,\cons)$ be a list object on an object $A$. Then, the following are equivalent:
  \begin{enumerate}
    \item $(L(A),\nil,\cons)$ is a symmetric list object;
    \item $(L(A),\nil,\App)$ is a commutative monoid object;
    \item $(L(A),\nil,\App,\eta)$ is the free commutative monoid object on $A$.
  \end{enumerate}
\end{theoremrep}

\begin{proof}
  $(1\Rightarrow2)$ We have already shown that list objects are monoids, so we need only show that $\App$ is commutative:
  \begin{equation*}
    %
  \end{equation*}

  Let the two maps above be $f$ and $g$ respectively. By definition, $f=\App=\iter(\id,\cons)$, so we show that the braided version satisfies the same computation rules, using the right unit law from Theorem~\ref{thm:free-monoid} and Lemma~\ref{lem:eta-app-braided}:

  $g$-$\nil$-$\beta$:
  \begin{equation*}
    %
  \end{equation*}

  $(2\Rightarrow1)$ $L(A)$ is a list object by assumption, so it has iterators for all list algebras, in particular those satisfying the ilst axiom, and it suffices to show the ilst axiom only:

  \begin{equation*}
    %
  \end{equation*}

  This establishes the equivalence of $(1)$ and $(2)$.

  $(3\Rightarrow 2)$ is clear, and since $(1)$ and $(2)$ are equivalent, it remains to show $(1\wedge2\Rightarrow3)$.

  $(1\wedge2\Rightarrow 3)$ $(L(A),\nil,\App)$ is a commutative monoid object by assumption, so we need only show freeness. Let $f:A\to B$ be a morphism, with $(B,e,m)$ a commutative monoid object. Since monoid homomorphisms are automatically commutative monoid homomorphisms, we may construct the extension $f^\sharp$ as $\iter(n,c)$, where $n$ and $c$ are:
  \begin{equation*}
    %
  \end{equation*}

  as in the non-commutative case. Since $L(A)$ is a list object, the rest of the proof of Theorem~\ref{thm:free-monoid} applies verbatim: $f^\sharp\circ\eta=f$, the homomorphism property via $\iter(f^\sharp,c)$, and uniqueness only iterate over list algebras with this same $c$. For the remark below, we also check that $c$ satisfies the ilst axiom:

  \begin{equation*}
    %
    \qedhere
  \end{equation*}
\end{proof}

Note the proofs of $(1\Rightarrow2)$ and $(1\wedge2\Rightarrow3)$ do not use that $L(A)$ is a list object: for a symmetric list object, Lemmas~\ref{lem:eta-app} and~\ref{lem:eta-app-braided}, and the monoid laws and freeness argument in the proof of Theorem~\ref{thm:free-monoid}, carry over unchanged, since every list algebra they iterate over satisfies the ilst axiom (for freeness, by the computation above). Hence every symmetric list object is the free commutative monoid on $A$.

\begin{theorem}
  If $\cat{C}$ admits parametrised symmetric list objects for every $A$,
  then $\mathsf{CMon}(\cat{C})$ is monadic over $\cat{C}$ via the evident forgetful functor.
\end{theorem}
\begin{proof}[Proof sketch]
  By the above, $A\mapsto L(A)$ is left adjoint to the forgetful functor $\mathsf{CMon}(\cat{C})\to\cat{C}$, and the proof for $\mathsf{Mon}(\cat{C})$ carries over, with one additional step. The monoid $(A,e,m)$ obtained from an algebra $\alpha:L(A)\to A$ is commutative, since $m=\alpha\circ\App\circ(\eta\otimes\eta)$ and $\App$ is commutative. Hence $m$ satisfies the ilst axiom, which is needed to identify $\alpha$ with the iterator $\iter(e,m)$.
\end{proof}

\begin{theoremrep}
  If $(\cat{C},\times,1)$ is cartesian monoidal and admits symmetric list objects for every $A$, then the induced monad $L:\cat{C}\to\cat{C}$ is strong, and further, commutative.
\end{theoremrep}
\begin{proof}[Proof sketch]
  The strong universal property (Theorem~\ref{thm:strong-up}) holds for symmetric list objects when restricted to maps $c$ satisfying $c(a,c(b,x,p),p)=c(b,c(a,x,p),p)$, by the same proof, since then $\check{c}$ satisfies the ilst axiom, as does $\pi_2:A\times P\to P$. The map $c$ defining $\tau$ satisfies this condition by the ilst axiom for $L(A\times B)$, so the construction of the strength maps is the same as in the non-symmetric case. The algebras used for the strength coherences are built from $\cons$ and $\App$, and satisfy the condition since $\App$ is commutative and associative. In this case, the strength also satisfies the following additional commutativity coherence:
  \begin{equation*}
    \tikzset{every picture/.style={}}
    \begin{tikzcd}[ampersand replacement=\&]
      \& {L(A)\times L(B)} \& \\
      {L\bigl(L(A)\times B\bigr)} \&\& {L\bigl(A\times L(B)\bigr)} \\
      {LL(A\times B)} \&\& {LL(A\times B)} \\
      \& {L(A\times B)}
      \arrow["{\sigma_{L(A),B}}"', from=1-2, to=2-1]
      \arrow["{\tau_{A,L(B)}}", from=1-2, to=2-3]
      \arrow["{L(\tau_{A,B})}"', from=2-1, to=3-1]
      \arrow["{\tau_{A,B}^\sharp}", from=2-1, to=4-2]
      \arrow["{L(\sigma_{A,B})}", from=2-3, to=3-3]
      \arrow["{\sigma_{A,B}^\sharp}"', from=2-3, to=4-2]
      \arrow["{\mu_{A\times B}}"', from=3-1, to=4-2]
      \arrow["{\mu_{A\times B}}", from=3-3, to=4-2]
    \end{tikzcd}
  \end{equation*}

  The lower pairs of edges compute to $\tau$ and $\sigma$ when precomposed with $\eta_{L(A)\times B}$ and $\eta_{A\times L(B)}$, respectively, and are thus extensions of the strength maps by the universal property of $(-)^\sharp$. It remains to show that the two sides of the upper diamond agree. Both are the doubly parametrised iterator $\checkiter(n,c):L(A)\times L(B)\to L(A\times B)$, with $n$ constantly $\nil$ and $c(a,l,ys)=\sigma_{A,B}(a,ys)\mathbin{\App}l$, which satisfies the restricted condition since $\App$ is commutative and associative. The path through $\tau_{A,L(B)}$ satisfies the computation rules directly, while for the path through $\sigma_{L(A),B}$, the $\cons$-$\beta$ rule requires an auxiliary induction over $L(B)$, again using that $\App$ is commutative and associative.
\end{proof}

Symmetric lists give an alternative way of constructing free commutative monoids in a symmetric monoidal category which is relevant to the study of the exponential modality in linear logic. In the future, we hope to compare this to other approaches, such as that of \textcite{melliesExplicitFormulaFree2018}.

\clearpage
\printbibliography

\end{document}